\documentclass[10pt,a4paper, reqno]{amsart}
\usepackage{mathrsfs}

\usepackage{amsmath,amsfonts,verbatim}
\usepackage{latexsym}
\usepackage{amssymb,leftidx}
\usepackage{extarrows}
\usepackage{overpic}
\usepackage{color}
\usepackage{epsfig}
\usepackage{subfigure}
\usepackage{tikz}
\usepackage{bbm}
\usepackage{tikz, caption}
\usepackage[font=small,labelfont=bf]{caption}

\usepackage[colorlinks=true,backref=page]{hyperref}
\hypersetup{urlcolor=blue, citecolor=blue}

\usepackage{amssymb}
\usepackage{mathrsfs}
\usepackage{amscd}
\usepackage{cite}

\newcommand\R{\mathbb{R}}

\newcommand\pa{\partial}

\usepackage{graphicx}
\usepackage{float}

\numberwithin{equation}{section}
\newtheorem{proposition}{Proposition}[section]

\newtheorem{lemma}{Lemma}[section]
\newtheorem{theorem}{Theorem}[section]

\newtheorem{remark}{Remark}[section]

\begin{document}
\title[Wave equation on some conical manifolds]{$L^p$-estimates for the wave equation with critical magnetic potential on conical manifolds}

\author{Xiaofen Gao}
\address{School of Mathematics and Statistics, Zhengzhou University, Zhengzhou, Henan, China, 450001; }
\email{gaoxiaofen@zzu.edu.cn}

\author{Jialu Wang}
\address{School of Mathematics and Statistics, Beijing Institute of Technology, Beijing, China, 100081; }
\email{jialu\_wang@bit.edu.cn}

\author{Chengbin Xu}
\address{School of Mathematics and Statistics, Qinghai Normal University, Xining, Qinghai, China, 810008; }
\email{xcbsph@163.com}

\author{Fang Zhang}
\address{Department of Mathematics, Beijing Institute of Technology, Beijing 100081}
\email{zhangfang@bit.edu.cn}

\begin{abstract}In this paper, we consider a class of conical singular spaces $\Sigma=(0,\infty)_r\times Y$ equipped with the metric $g=\mathrm{d}r^2+r^2h$, where the cross section $Y$ is a compact $(n-1)$-dimensional closed Riemannian manifold $(Y,h)$ without boundary. In this context, we aim to show that the sine wave propagator $\sin\left(t\sqrt{\mathcal{L}_{\mathbf{A}}}\right)/\sqrt{\mathcal{L}_{\mathbf{A}}}$ is bounded in $L^{p}(\Sigma)$, where $\mathcal{L}_{\mathbf{A}}$ is a magnetic Schr\"odinger operator with a scaling-critical magnetic potential on metric cone $\Sigma$. Our main result is the generalization of the result in \cite{L}. The novel ingredient is the
construction of Hadamard parametrix for $\cos\left(t\sqrt{\mathcal{L}_{\bf A}}\right)$ on $\Sigma$.

\end{abstract}

 \maketitle

\begin{center}
 \begin{minipage}{120mm}
 \end{minipage}
 \end{center}

 \tableofcontents

\section{Introduction}
In the present paper, we continue our previous program in \cite{WZZ} to investigate the $L^p$-estimates for the following wave equations perturbed by critical magnetic potentials on conical singular space $\Sigma$
\begin{align}\label{LS}
\begin{cases}
\pa_{tt} u+\mathcal{L}_{\bf A}u=0,\quad & (t,x)\in \R\times \Sigma,\\
u(0,x)=0,\quad &\pa_t u(0,x)=g(x),\quad x\in \Sigma.
\end{cases}
\end{align}
More precisely, under the restricted  condition of the injectivity radius $\rho>\pi$, we have proved the  $L^p$-estimates  of the wave operators with critical magnetic potentials hold for $|\frac1p-\frac12|<\frac1{n-1}$ in the metric cone setting, which extends the results in \cite{L}.

\subsection{The setting and motivation } The metric cone $\Sigma=C(Y)$ on $Y$ is a product space $(0,\infty)_r\times Y$ endowed with the metric $g=dr^2+r^2h$, where $Y$ is an $n-1$-dimensional compact connected Riemannian manifold without boundary satisfying $n-1\geq2$. The setting cone $\Sigma$ is a simple conical singular space, with the Euclidean space $\R^n$ serving as its most straightforward example where the cross section is $Y=\mathbb{S}^{n-1}$ equipped with the standard metric $d\theta^2$. Notably, general metric cones exhibit a dilation symmetry analogous to that of Euclidean space, but typically lack any other symmetries. 
On metric cone $\Sigma$, the operator $\mathcal{L}_{\bf A}$, perturbed by magnetic potential, is a conical singular operator given by
\begin{equation}\label{op-L}
\mathcal{L}_{\bf A}=-\pa_r^2-\frac{n-1}{r}\pa_r+\frac{(i\nabla_y+{\bf A}(y))^2}{r^2},
\end{equation}
where ${\bf A}: C^\infty(Y)\to\R^{n-1}$, and $(i\nabla_y+{\bf A})^2$ denotes the Schr\"odinger operator with magnetic potentials on $Y$. We consider its Friedrichs extension in this paper. It's worth remarking that the magnetic potential studied here is doubly critical, which are the scaling invariance and singularity of the potential. These features introduce additional  difficulties in various research fields, see \cite{FFFP,FFFP1,FZZ,GYZZ} and reference therein for dispersive and Strichatz estimates of Schr\"odinger, wave and Kelin-Gordorn equations; see \cite{BPST,FZZ1} for uniform resolvent estimates.
The operator $\mathcal{L}_{\bf A}$ without potential, known as positive Laplace-Beltrami operator $\Delta_g$, has received extensive attentions from different mathematical perspectives, contributing significant results. For example, wave diffraction phenomenon has been considered in \cite{CT} by Cheeger and Taylor; Riesz transform and heat kernel have been found in \cite{HR,LJFA,LBSM,MJFA}; Blair-Ford-Marzuola \cite{BFM,F} derived Strichartz estimate for flat cone $C(\mathbb{S}^1_\rho)$, and then Zhang-Zheng \cite{ZZ2017,ZZ2020} subsequently generalized this result to metric cones; Additionally, Li \cite{LMZ} and M\"uller-Seeger \cite{MS1} have established the $L^p$ regularity estimate, which is one of motivation in our paper.
\vspace{0.2cm}

The main goal of studying the $L^p$-estimates for the wave operators in flat Euclidean space $\R^n$ is to determine the range of $p$ for which the following inequality holds
\begin{equation}\label{est-sin}
\left\|\frac{\sin(t\sqrt{-\Delta})}{\sqrt{-\Delta}}g\right\|_{L^p(\R^n)}\leq C_{p,t} \left\|g\right\|_{L^p(\R^n)},\quad \forall g\in L^p(\R^n),
\end{equation}
where $C_{p,t}$ is a constant depending on two parameters $p$ and time $t$.
This inequality was first studied by Sj\"ostrand \cite{S}, in which \eqref{est-sin} was valid for $p_0<p<p_0'$, but failed when $p<p_0$ or $p>p_0$ with $p_0$ being the critical index. However, these results did not contain the endpoint case, that is $p=p_0$ or $p_0'$. Fortunately, Persal \cite{P} and Miyachi \cite{M} solved this remaining issues, and so obtained that \eqref{est-sin} holds for the sharp range of $p$, namely $|\frac1p-\frac12|\leq \frac1 {n-1}$. The sharp $L^p$-estimates are often applied to study the almost everywhere convergence problems, which are central issues in the field of harmonic analysis, see \cite{PF} and the references therein for more details.

\vspace{0.2cm}

Considering the significant research implications of $L^p$-estimates, it seems natural to seek an analogue of \eqref{est-sin} on other manifolds or Schr\"odinger operators with potentials. Actually, over the past few decades, many mathematicians have studied the $L^p$-estimates for Schr\"odiger operators with potential, making it impractical to cite all relevant works here. For the Hermite operator $-\Delta+|x|^2$, related results were presented in \cite{NT}, and more general operators of the form $-\Delta+V$ were investigated by Zhong \cite{Z}. For the  Grushin operator $G=-\Delta-|x|^2\partial_{t}^2$, K. Jotsaroop and S. Thangavelu \cite{KT} established the $L^p$-estimates by proving that the operator $\frac{\sin\sqrt{G}}{G}$ is bounded on $L^p(\R^{n+1})$ for every $p$ satisfying $|\frac{1}{p}-\frac{1}{2}|<\frac{1}{n+2}$. Additionally, in the subsequent work \cite{TN}, they derived the $L^p\rightarrow L^2$ estimate for wave equation with the Grushin operator via employing the $L^p$ boundedness of Riesz transform. We also emphasize the Schr\"odinger operators perturbed by a scaling-critical potential, such as Aharonov-Bohm magnetic potential and inverse-square electric potentials \cite{FFFP}, which  have attracted considerable research interest in recent years. Nevertheless, establishing $L^p$-estimates is not a straightforward work because of the doubly critical properties as above stated. Very recently, by constructing the kernel function and proving the pointwise estimate of an analytic operator family, Wang et al. \cite{WZZ} demonstrated the $L^p$-estimates for the Schr\"odinger operators with Aharonov-Bohm magnetic potential in $\R^2$. And they applied this result to show the $L^p$-boundedness of the operators $(I+\mathcal{L}_{\bf A})^{-\gamma}e^{it\sqrt\mathcal{L}_{\bf A}}$ for $1<p<\infty$, when $\gamma>|\frac1p-\frac12|$. We also refer to \cite{JZ1,JZ2} for Schr\"odinger operators with critical electric potential in higher dimensions.\vspace{0.2cm}

There are also numerous literature studying $L^p$-estimates on various manifolds, including the metric cone space considered in this paper. Specially, Li and Lohoue \cite{LN} established $L^p$-estimates for the wave equation in the setting, assuming that $\rho>\pi$, where $\rho$ is the injective radius of the closed section $Y$. Later, employing the properties of spectral multipliers in the Riemannian manifold framework, Li \cite{L} improved this result by removing the restricted conditions on $\rho$, and then obtained
\begin{equation}
\|(I+\Delta_g)^{-\frac\gamma2}e^{it\sqrt{\Delta_g}}f\|_{L^p(\Sigma)}\leq C(p,\gamma)(1+t)^{\gamma}\|f\|_{L^p(\Sigma)},
\end{equation}
for $1<p<\infty$ satisfying $(n-1)|\frac1p-\frac12|<\gamma$, where $\Delta_g$ is the positive Laplace-Beltrami operator. It is important to note that the results in \cite{LN, L} are valid for $|\frac1p-\frac12|<\frac1{n-1}$, except the endpoint $|\frac1p-\frac12|=\frac1{n-1}$. Recently, the end point case for dimension $n=3$ has been resolved by Pan and Fan \cite{PF}. However, they were unable to demonstrate the optimum result at the endpoint if $n\neq3$. Moreover, $L^p$-estimates for the wave equation have been extensively investigated in other various settings, including Heisenberg groups as explored in \cite{MS2}, symmetric spaces of non-compact type with real rank one as presented in \cite{GM, Ionescu}, and Damek-Ricci spaces as discussed in \cite{MV}.
For extra results on the wave equation in manifolds with conical singularities, refer to \cite{Mewun,MS1,SS1,SS2}. Motivated by the aforementioned papers and following the argument in \cite{LN,L}, we aim to establish the $L^p$-estimates for the wave equation \eqref{LS}, involving a scaling critical magnetic Schr\"odinger operator $\mathcal{L}_{\bf A}$ on the metric cone manifold. More precisely, we have the following main results.

\subsection{The main results}

In this paper, we concern the Cauchy problem for the wave equation on the metric cone
\begin{align}\label{LS}
\begin{cases}
\pa_{tt} u+\mathcal{L}_{\bf A}u=0,\quad & (t,x)\in \R\times C(Y),\\
u(0,x)=0,\quad &\pa_t u(0,x)=g(x),\quad x\in C(Y)
\end{cases}
\end{align}
where $\mathcal{L}_{\bf A}$ is defined by \eqref{op-L}, and represents the Schr\"odinger operator perturbed by magnetic potential on the cone. By spectral theory, the solution $u(t,x)$ can be written as
$$u(t,x)=\frac{\sin(t\sqrt{\mathcal{L}_{\bf A}})}{\sqrt{\mathcal{L}_{\bf A}}}g(x)
,\quad x\in C(Y),\ t>0.$$
Our main results are as follows.
\begin{theorem}\label{est-wave}
Let $Y$ be a compact Riemannian manifold without boundary, codimension $n-1\geq2$, and the injectivity radius $\rho>\pi$. Then, for $|\frac1p-\frac12|<\frac1{n-1}$ with $1<p<+\infty,$ there exists a constant $C_{p}>0$ such that
$$\Big\|\frac{\sin(t\sqrt{\mathcal{L}_{\bf A}})}{\sqrt{\mathcal{L}_{\bf A}}}g\Big\|_{L^p(C(Y))}\leq C_{p} t\|g\|_{L^p(C(Y))}.$$
\end{theorem}

Now we give the strategy of proof here. For fixed $t>0$, we denote the family of analytic operators on $C(Y)$:
\begin{align}\label{equ:fanaoper}
F_{\omega,t}(\mathcal{L}_{\bf A})=(\pi/2)^{\frac12}(t\sqrt{\mathcal{L}_{\bf A}})^{\omega-\frac{n}2}J_{\frac{n}2-\omega}
         (t\sqrt{\mathcal{L}_{\bf A}}),
 \end{align}
where $\omega=\epsilon +s(\frac{n+1}2-\epsilon)i$ with $s\in\R.$
Similar to the approach in \cite{L, LN}, we make use of a combination of the family of analytic operators (see \eqref{equ:fanaoper} above) with Stein's interpolation theorem to establish Theorem \ref{est-wave}.
Then the problem can be reduced to verify the following result:
\begin{theorem}\label{theo-Lp}
Let $1\leq p\leq+\infty$ and $\frac12<\epsilon<1$. Then, there exists a constant $C_{\epsilon}>0$ such that for $t>0$ and for all $g\in L^p(C(Y))$, there holds
\begin{align}
\|F_{\omega,t}(\mathcal{L}_{\bf A})g\|_{L^p}\leq C_{\epsilon}\|g\|_{L^p},
\end{align}
where $\omega=\epsilon +s(\frac{n+1}2-\epsilon)i$ with $s\in\R.$
\end{theorem}
In the proof process of Theorem \ref{theo-Lp}, the properties of Legendre functions play a fundamental role, and the novel ingredient is the construction of the Hadamard parametrix for $\cos(t\sqrt{\mathcal{L}_{\bf A}})$.


\begin{remark} In this paper, we are unable to obtain that the main result holds for $\rho<\pi$ due to the lack of heat kernel estimates for $\mathcal{L}_{\bf A}$, as well as for the endpoint case $|\frac1p-\frac12|=\frac1{n-1}$. Thus, we leave these as open questions for our future study.
\end{remark}

\subsection{The structure of the paper} This paper is organized as follows. Section \ref{Preliminaries} is devoted to presenting the kernel of operator $F_{\omega,t}(\mathcal{L}_{\bf A})$, utilizing the formula for computing the kernel of $\varphi(-\Delta_g)$ on the conical manifold, as outlined in \cite[p.173]{T86} and \cite[p.276]{CT82}. Section \ref{proof-2} contains the proof of Theorem \ref{theo-Lp} and Theorem \ref{est-wave}. Meanwhile, the proof of some crucial Lemmas and Proposition \ref{lemT}, established in Section \ref{proof-2}, is provided in Section \ref{proof12} and Section \ref{proof} respectively. In the Appendix, we construct the Hadamard parametrix for the Schr\"odinger operator with scaling-critical magnetic potential on the compact Riemannian manifold.

\section{Preliminaries}\label{Preliminaries}

This section mainly provides the representation of the kernel for the operator $F_{w,t}(\mathcal{L}_{\bf A})$, which will be essential in the following proof process. We begin with introducing some notations throughout this paper.

\subsection{Notations}  We utilize $C$ to denote the universal constant, which may change from one line to another. If a constant $C$ depends on a special parameter, we shall denote it explicitly by subscripts. For instance,
$C_\epsilon$ should be understood as a positive constant
depending on $\epsilon$. For convenience, we also use $A\lesssim B$ to denote the statement that $A\leq CB$ for some large constant $C$ which may vary
from line to line. In particular, we make use of $A\lesssim_\epsilon B$ to denote $A\leq C B$ with $C$ depending on parameter $\epsilon$.
In addition, we will adopt the notations for special functions as outlined in \cite{W}. For example, we employ the Gamma function by $\Gamma(z)$, and the Legendre functions by $P_\lambda^\mu(z)$ with $|z|<1$ and $Q_\lambda^\mu(z)$ with $|z|>1$. Throughout
this paper,  the operator on $Y$ is defined by
\begin{equation}\nu=\sqrt{(i\nabla_y+{\bf A})^2+\beta^2},\end{equation}
where
\begin{equation}
\beta=\tfrac{n-2}2,\qquad \omega=\epsilon+is\Big(\frac{n+1}2-\epsilon\Big),
\end{equation}
and $\frac12<\epsilon<1$,\  $s\in\R$.

We denote $d$(resp. $d_Y)$ to be the Riemannian distance induced by $\Sigma$ (resp. $Y$),  and we also denote $d\mu(r,x)$ (resp. $d\mu_Y(y)$) by the  Riemannian measure induced by $\Sigma$ (resp. $Y$). We obviously have:
\begin{equation}\label{equ:dmrx}
  d\mu(r,x)=r^{n-1}dr\;d\mu_Y(y).
\end{equation}



\subsection{The kernel of the operator $F_{w,t}(\mathcal{L}_{\bf A})$} In this subsection, for the purpose of calculating the kernel of operator $F_{w,t}(\mathcal{L}_{\bf A})$, we begin with recalling the following facts:

\underline{\textbf{ I:}} Let $\mathrm{Re}\mu\neq-\frac{1}{2}, \mathrm{Re} \lambda\neq-\frac{1}{2}$ and $a, b, c>0$, and define
$$A= \arccos\Big(\frac{b^{2}+c^{2}-a^{2}}{2bc}\Big),\quad \mathcal{A}=\cosh^{-1}\Big(\frac{a^{2}-b^{2}-c^{2}}{2bc}\Big).$$
Then, by \cite[p.412]{W}, we have
\begin{align}\label{formula:Ma}
&\int_{0}^{+\infty} t^{1-\mu} J_{\mu}(at) J_{\lambda}(bt) J_{\lambda}(ct) \, dt \nonumber\\
&=
\begin{cases}
0, & \text{if } a < |b - c|, \\
\frac{(bc)^{\mu - 1} \sin^{\mu - \frac{1}{2}} A}{(2\pi)^{\frac12}  a^{\mu}} P_{\lambda - \frac{1}{2}}^{\frac{1}{2} - \mu}(\cos A), & \text{if } |b - c| < a < b + c, \\
\frac{(bc)^{\mu - 1} \sinh^{\mu - \frac{1}{2}} \mathcal{A}}{\left( \frac{\pi^{3}}{2} \right)^{\frac{1}{2}} a^{\mu}} \cos(\pi \lambda) Q_{\lambda - \frac{1}{2}}^{\frac{1}{2} - \mu}(\cosh \mathcal{A}), & \text{if } a > b + c.
\end{cases}
\end{align}

\underline{\textbf{ II:}} Let $ 0<\theta<\pi$ and $\mathrm{Re} \mu_{*}<\frac{1}{2}$. Then, by \cite[p.188]{MOS66}, we can see
\begin{equation}\label{gamma}
\begin{split}
 \Gamma\Big(\frac{1}{2}-\mu_{*}\Big)&P_{\lambda}^{\mu_{*}}(\cos\theta)\\
&=\Big(\frac{\pi}{2}\Big)^{-\frac{1}{2}}(\sin\theta)^{\mu_{*}}\int_{0}^{\theta}(\cos s-\cos\theta)^{-\mu_{*}-\frac{1}{2}}\cos[(\lambda+\frac{1}{2})s]\,ds.
\end{split}
\end{equation}

\underline{\textbf{ III:}} Let $\mathrm{Re}(\lambda+\mu_{*}+1)>0$ and $\mathrm{Re}\mu_{*}<\frac{1}{2}$. By \cite[p.186]{MOS66}, one has
\begin{align}\label{gammaQ}
\Gamma\Big(\frac{1}{2}-\mu_{*}\Big)&Q_{\lambda}^{\mu_{*}}(\cosh\theta)\nonumber
\\&=\Big(\frac{\pi}{2}\Big)^{\frac{1}{2}}e^{i\pi\mu_{*}}(\sinh\theta)^{\mu_{*}}\int_{\theta}^{+\infty}(\cosh s-\cosh\theta)^{-\mu_{*}-\frac{1}{2}}e^{-(\lambda+\frac{1}{2})s}\,ds.
\end{align}

We denote $K_{w,t}\big((r,y),(r^\prime,y^\prime)\big)$ as the kernel of the operator $F_{w,t}(\mathcal{L}_{\bf A})$ with $(r,y),(r^\prime,y^\prime)\in C(Y)$. Consequently, following the similar argument in \cite[p.173]{T86} or \cite[p.276]{CT82} for the kernel formulas of the operator $f(-\Delta_g)$ on the conical manifold, and recalling $F_{\omega,t}(\mathcal{L}_{\bf A})$ in \eqref{equ:fanaoper}, we can obtain
\begin{align*}
 K_{\omega,t}\big((r,y),(r^\prime,y^\prime)\big)&:=F_{w,t}(\mathcal{L}_A)=(rr^\prime)^{-\beta}\int_0^{+\infty}f_{\omega,t}(\lambda^2)J_\nu(\lambda r)J_\nu(\lambda r^\prime)\lambda\;d\lambda\\\nonumber
&=\big(\tfrac{\pi}2\big)^\frac12t^{\omega-\frac{n}{2}}(rr^\prime)^{-\beta}\int_0^{+\infty}\lambda^{1-(\frac{n}{2}-\omega)}J_{\frac{n}{2}-\omega}(t\lambda )J_\nu(\lambda r)J_\nu(\lambda r^\prime)\;d\lambda.
\end{align*}
Then combining \eqref{formula:Ma}, \eqref{gamma}, with \eqref{gammaQ}, we deduce that, when ${\rm Re}\big(\tfrac{n}2-w\big)>0$, there holds
\begin{align}\label{kernel-op}
& K_{\omega,t}\big((r,y),(r^\prime,y^\prime)\big)\\\nonumber
=&\frac12t^{2(w-\frac{n}{2})}(rr^\prime)^{-\omega}\times \begin{cases}
0,\qquad \text{if}\quad t<|r-r^\prime|;\\
\sin^{\frac{n-1}{2}-\omega}A P_{\nu-\frac12}^{\omega-\frac{n-1}{2}}(\cos A),\quad \text{if}\quad |r-r^\prime|<t<r+r^\prime;\\
\frac{2}{\pi}\sinh^{\frac{n-1}{2}-\omega}\mathcal{A}\cos(\pi\nu)Q_{\nu-\frac12}^{\omega-\frac{n-1}{2}}(\cosh \mathcal{A}),\quad \text{if}\quad t>r+r^\prime; \end{cases}\\\nonumber
=&\frac{t^{2(\omega-\frac{n}{2})}(rr^\prime)^{-\omega}}{\sqrt{2\pi}\Gamma(\frac{n}{2}-\omega)}
\times  \begin{cases}
0,\qquad \text{if}\quad t<|r-r^\prime|;\\
\int_0^A (\cos h-\cos A)^{\beta-\omega}\cos h\nu(y,y^\prime) \;dh,\quad \text{if}\quad |r-r^\prime|<t<r+r^\prime;\\
e^{i\pi(\omega-\frac{n-1}{2})}\int_{\mathcal{A}}^{+\infty}(\cosh h-\cosh \mathcal{A})^{\beta-\omega}e^{-h\nu }\cos\nu\pi(y,y^\prime)\;dh,\quad \text{if}\quad t>r+r^\prime; \end{cases}
\end{align}
where
\begin{equation}\label{equ:amathcal2}
A=\arccos\frac{r^2+{r^\prime}^2-t^2}{2rr^\prime}\quad \text{and}\quad \mathcal{A}=\cosh^{-1}\frac{t^2-r^2-{r^\prime}^2}{2rr^\prime}.
\end{equation}

\section{Proofs of main results}\label{proof-2}

In this section, we first give the proof of Theorem \ref{theo-Lp}, and then use Theorem \ref{theo-Lp} to establish Theorem \ref{est-wave}.
For the purpose, we define two integral operators $T_1=T_1(\omega,t)$ and $T_2=T_2(\omega,t)$ respectively to be equal to
\begin{align}\label{T1}
t^{2(w-\frac n2)}(rr^\prime)^{-\omega}\mathbbm{1}_{\{t>r+r^\prime\}}\int_{\mathcal{A}}^{+\infty}(\cosh h-\cosh\mathcal{A})^{\beta-\omega}[e^{-\nu h}\cos\pi\nu]\;dh,
\end{align}
and
\begin{equation}\label{T21}
t^{2(\omega-\frac n2)}(rr^\prime)^{-\omega}\mathbbm{1}_{\{|r-r^\prime|<t<r+r^\prime\}}\int_0^A(\cos h-\cos A)^{\beta-\omega}\cos (h\nu) (y,y^\prime)\;dh.
\end{equation}
Then, by \eqref{kernel-op}, the following relation holds
\begin{equation}\label{div}
F_{\omega,t}(\mathcal{L}_{\bf A})g(r,y)=\frac{1}{\sqrt{2\pi}\Gamma(\frac{n}{2}-\omega)}[e^{i\pi(\omega-\frac{n-1}2)}T_1+T_2]g(r,y),
\end{equation}
where
\begin{align*}
T_{i}(\omega,t)g(r,y)=\int_{C(Y)}T_{i}(\omega,t)((r,y),(r^\prime,y^\prime))g(r^\prime,y^\prime)\;d\mu(r^\prime,y^\prime),\quad \forall g(r^\prime,y^\prime)\in L^p(C(Y)),
\end{align*}
with $i=1,2$.

Hence, to show Theorem \ref{theo-Lp}, it suffices to prove the following propositions:
\begin{proposition}\label{prop:T1}
For each $1\leq p\leq+\infty,$ there exists a constant $C_\epsilon>0$ such that
\begin{equation}\label{prop-ine-T1}
\|T_1(\omega,t)g\|_p\leq C_\epsilon\|g\|_p,\quad \forall g\in L^p(C(Y)).
\end{equation}
\end{proposition}
\begin{proposition}\label{prop:T2}
For each $1\leq p\leq+\infty,$ there exists a constant $C_{\epsilon}>0$ such that
\begin{equation}\label{prop-ine-T2}
\|T_{2}(\omega,t)g\|_p\leq C_{\epsilon}\|g\|_p,\quad \forall g\in L^p(C(Y)).
\end{equation}
\end{proposition}

\subsection{Proof of Proposition \ref{prop:T1}} In this subsection, we focus on establishing Proposition \ref{prop:T1}. To achieve this goal, we highly rely on two key results, Lemma \ref{lem:kerest1} and Lemma \ref{lem:kerest2}, which will be showed in subsection \ref{proof12}.
\begin{lemma}\label{lem:kerest1} Let $H(y,y^\prime)$ be the kernel of the operator $e^{-\nu h}\cos(\pi\nu)$ and denote
\begin{equation}\label{def-H1}
\|e^{-h\nu}\cos(\pi\nu)\|_*:=\sup_{y\in Y}\int_Y|H(y,y^\prime)|\;d\mu_Y(y^\prime)
\end{equation}
and
\begin{equation}\label{def-H2}
\|e^{-h\nu}\cos(\pi\nu)\|^*:=\sup_{y^\prime\in Y}\int_Y|H(y,y^\prime)|\;d\mu_Y(y).
\end{equation}
 Then, there exists a constant $C>0$ such that
\begin{align}\label{ine:kere1}
\|e^{-h\nu}\cos(\pi\nu)\|_*+\|e^{-h\nu}\cos(\pi\nu)\|^*\leq Ce^{-h\beta}(1+h^{\frac {1-n}2}).
\end{align}
\end{lemma}
\begin{lemma}\label{lem:kerest2}Let $\|e^{-h\nu}\cos(\pi\nu)\|_*$ and $\|e^{-h\nu}\cos(\pi\nu)\|^*$ be in \eqref{def-H1} and \eqref{def-H2} respectively, then  there exists a constant $C_\epsilon>0$ satisfying
\begin{align}\label{est-ps}
&\int_{\mathcal{A}}^{+\infty}(\|e^{-h\nu}\cos(\pi\nu)\|_*+\|e^{-h\nu}\cos\pi\nu\|^*)|(\cosh h-\cosh\mathcal{A})^{\beta-w}|\;dh\nonumber\\
\leq& C_\epsilon(\cosh\mathcal{A}-1)^{-\epsilon},\quad\forall\mathcal{A}>0.
\end{align}
\end{lemma}
We postpone the proof of Lemma \ref{lem:kerest1} and Lemma \ref{lem:kerest2} to subsection \ref{proof12}.
With Lemma \ref{lem:kerest2} in hand, we now turn our attentions to showing Proposition \ref{prop:T1} by recalling the representation of $T_1(w,t)((r,y),(r^\prime,y^\prime))$ in \eqref{T1}.
To proceed, we apply the estimate \eqref{est-ps} and the fact $\mathcal{A}=\cosh^{-1}\frac{t^2-r^2-{r^\prime}^2}{2rr^\prime}$ to derive that

\begin{align}
&\int_{C(Y)}|T_1((r,y),(r^\prime,y^\prime))|\;d\mu(r^\prime,y^\prime)\nonumber\\
=&\int_0^{+\infty}\int_Y\Big|t^{2(w-\frac n2)}(rr^\prime)^{-w}\mathbbm{1}_{\{t>r+r^\prime\}}e^{i\pi(w-\frac{n-1}2)}\nonumber\\
&\times\int_{\mathcal{A}}^{+\infty}(\cosh h-\cosh\mathcal{A})^{\beta-w}[e^{-\nu h}\cos\pi\nu](y,y^\prime)\;dh\Big|\;d\mu_Y(y^\prime){r^\prime}^{n-1}\;dr^\prime\nonumber\\
\lesssim&\int_0^{+\infty}t^{2(\epsilon-\frac n2)}(rr^\prime)^{-\epsilon}\Big[\int_{\mathcal{A}}^{+\infty}\|e^{-h\nu}\cos\pi\nu\|_*|(\cosh h-\cosh\mathcal{A})^{\beta-w}|\;dh\Big]\mathbbm{1}_{\{t>r+r^\prime\}}{r^\prime}^{n-1}\;dr^\prime
\nonumber\\
\lesssim_\epsilon&\int_0^{+\infty}t^{2(\epsilon-\frac n2)}(rr^\prime)^{-\epsilon}(\cosh\mathcal{A}-1)^{-\epsilon}\mathbbm{1}_{\{t>r+r^\prime\}}{r^\prime}^{n-1}\;dr^\prime\nonumber\\
=_\epsilon&\int_0^{+\infty}t^{2(\epsilon-\frac n2)}[t^2-(r+r^\prime)^2]^{-\epsilon}\mathbbm{1}_{\{t>r+r^\prime\}}{r^\prime}^{n-1}\;dr^\prime.\nonumber\\
\lesssim_\epsilon& t^{2(\epsilon-\frac n2)}t^{n-1}t^{-\epsilon}\int_0^{t-r}[t-(r+r^\prime)]^{-\epsilon}\;dr^\prime
\lesssim_\epsilon1.\nonumber
\end{align}
Hence, it concludes that there exists a constant $C_\epsilon$ satisfying
\begin{align}\label{est-rx}
\sup_{(r,y)\in C(Y)}\int_{C(Y)}|T_1((r,y),(r^\prime,y^\prime))|\;d\mu(r^\prime,y^\prime)< C_\epsilon
\end{align}
for $t>0$, $\omega=\epsilon +is(\frac{n+1}2-\epsilon)$ with $s\in\R$, $\frac12<\epsilon<1$.

Since
$$|T_1((r,y),(r^\prime,y^\prime))|=|T_1((r^\prime,y^\prime),(r,y))|,$$
there also holds
\begin{align} \label{est-rrx}
\sup_{(r^\prime,y^\prime)\in C(Y)}\int_{C(Y)}|T_1((r,y),(r^\prime,y^\prime))|\;d\mu(r,y)< C_\epsilon.
\end{align}
Using Riesz convexity theorem, and so Proposition \ref{prop:T1} follows.

\subsection{Proof of Proposition \ref{prop:T2}} In this subsection, we need the following kernel estimate for  $T_{2}((r,y),(r^\prime,y^\prime))$, which will allow us to establish Proposition \ref{prop:T2}.
\begin{proposition}\label{lemT}
Let $T_{2}((r,y),(r^\prime,y^\prime))$ be in \eqref{T21}, and $\frac12<\epsilon<1$, then there exists a constant $C_{\epsilon}>0$ such that
\begin{align}\label{est-KT21}
|T_{2}((r,y),&(r^\prime,y^\prime))|\leq C_{\epsilon}t^{2(\epsilon-\frac n2)}(rr^\prime)^{-\epsilon}\mathbbm{1}_{\{y,y^\prime;d_Y(y,y^\prime)<A\}}\mathbbm{1}_{\{r,r^\prime>0;|r-r^\prime|<t<r+r^\prime\}}\\
&\times
\begin{cases}
(A^2-d_{Y}^2(y,y^\prime))^{-\epsilon}\Big[1+\frac{(A^2-d_{Y}^2(y,y^\prime))^{\frac12}}{d_{Y}(y,y^\prime)}\Big],&0<A<\frac\pi2,\nonumber\\
(A-d_{Y}(y,y^\prime))^{-\epsilon}(\pi-A)^{-\frac12-\epsilon}\Big[1+\frac{(\pi-A)(A-d_{Y}(y,y^\prime))^{\frac12}}{d_{Y}(y,y^\prime)}\Big],&\frac\pi2\leq A<\pi.\nonumber
\end{cases}
\end{align}
\end{proposition}
We postpone the proof of Proposition \ref{lemT} to Section \ref{proof}, and now utilize this Proposition to verify Proposition \ref{prop:T2}.\vspace{0.2cm}

{\bf\emph{The proof of Proposition \ref{prop:T2}}:}  First, we define
\begin{align}\label{def-K1}
K_1(&(r,y),(r^\prime,y^\prime)):=\mathbbm{1}_{\{y,y^\prime: d_Y(y,y^\prime)<A\}}\mathbbm{1}_{\{r,r^\prime>0;|r-r^\prime|<t<r+r^\prime\}}t^{2(\epsilon-\frac n2)}(rr^\prime)^{-\epsilon}\\
&\times\Big[(A^2-d_Y^2(y,y^\prime))^{-\epsilon}+\frac1{d_Y(y,y^\prime)}(A^2-d_Y^2(y,y^\prime))^{\frac12-\epsilon}\Big]
\times \mathbbm{1}_{\{\cos A=\frac{r^2+{r^\prime}^2-t^2}{2rr^\prime}>0,0<A<\frac\pi2\}},\nonumber
\end{align}
and
\begin{align}\label{def-K2}
K_2(&(r,y),(r^\prime,y^\prime)):=\mathbbm{1}_{\{y,y^\prime;d_Y(y,y^\prime)<A\}}\mathbbm{1}_{\{r,r^\prime>0;|r-r^\prime|<t<r+r^\prime\}}t^{2(\epsilon-\frac n2)}(rr^\prime)^{-\epsilon}\nonumber\\
&\times\Big[(\pi-A)^{-\frac12-\epsilon}(A-d_Y(y,y^\prime))^{-\epsilon}+\frac{(\pi-A)^{\frac12-\epsilon}}{d_Y(y,y^\prime)}(A-d_Y(y,y^\prime))^{\frac12-\epsilon}\Big]\nonumber\\
&\times \mathbbm{1}_{\{\cos A=\frac{r^2+{r^\prime}^2-t^2}{2rr^\prime}<0,\frac\pi2\leq A<\pi\}}.
\end{align}
Then it follows from \eqref{est-KT21} that there exists a constant $C$ such that
\begin{align}
T_2((r,y),(r^\prime,y^\prime))\leq C(K_1+K_2)((r,y),(r^\prime,y^\prime))
\end{align}
for all $d_Y(y,y^\prime)>0$.
Therefore, to prove Proposition \ref{prop:T2}, it is sufficient to verify the following lemma from Riesz convexity theorem:
\begin{lemma}\label{lemK}
Let $\omega\in\mathbb{C}$ with $\Re\omega=\epsilon$, where $\epsilon\in(\frac12,1)$. Then there exists a constant $C_{n,\epsilon}$ satisfying
\begin{align}\label{est-K1}
&\sup_{(r,y)\in C(Y)}\int_{C(Y)}|K_i((r,y),(r^\prime,y^\prime))|d\mu(r^\prime,y^\prime)<C_\epsilon,\qquad i=1,2,\\  \label{est-K2}
&\sup_{(r^\prime,y^\prime)\in C(Y)}\int_{C(Y)}|K_i((r,y),(r^\prime,y^\prime))|d\mu(r,y)<C_\epsilon,\qquad i=1,2.
\end{align}
\end{lemma}
\begin{proof}Since
$$
\int_{C(Y)}|K_i((r,y),(r^\prime,y^\prime))|d\mu(r^\prime,y^\prime)=\int_{C(Y)}|K_i((r,y),(r^\prime,y^\prime))|d\mu(r,y),
$$
we only need to check that \eqref{est-K1} holds for $i=1,2$.

{\bf{\emph{The proof for $i=1$}}}: First of all, using the variable change $h=As$, one has
\begin{align*}
&\int_Y\Big[(A^2-d_Y^2(y,y^\prime))^{-\epsilon}+\frac1{d_Y(y,y^\prime)}(A^2-d_Y^2(y,y^\prime))^{\frac12-\epsilon}\Big]\mathbbm{1}_{\{y,y^\prime;d_Y(y,y^\prime)<A\}}d\mu_Y(y^\prime)\\
\lesssim&\int_0^A\Big[(A^2-h^2)^{-\epsilon}+\frac1h(A^2-h^2)^{\frac12-\epsilon}\Big]h^{n-2}\;dh\\
=&A^{n-1-2\epsilon}\int_0^1\Big[(1-s^2)^{-\epsilon}s^{n-2}+(1-s^2)^{\frac12-\epsilon}s^{n-3}\Big]\;ds\\
\lesssim_{\epsilon}& A^{n-1-2\epsilon}
\end{align*}
for $n\geq 3$ and $\epsilon\in(\frac12,1)$. Note that
$$2\sin^2\frac A2=1-\cos A=1-\frac{r^2+{r^\prime}^2-t^2}{2rr^\prime},\quad A\in(0,\frac\pi2),$$
which implies
$$A^2\sim\frac{t^2-(r-r^\prime)^2}{2rr^\prime}.$$
Therefore, we only need to show
\begin{equation}\label{re-K1}
t^{2(\epsilon-\frac n2)}\int_{\{r,r^\prime:|r-r^\prime|<t<r+r^\prime\}}\Big[\frac{t^2-(r-r^\prime)^2}{2rr^\prime}\Big]^{\frac{n-1}2}[t^2-(r-r^\prime)^2]^{-\epsilon}{r^\prime}^{n-1}\;
dr^\prime\lesssim_\epsilon1.
\end{equation}
In the following, we divide into two cases : $r\geq t$ and $r< t$.

\emph{Case a}: $r\geq t$. A simple computation yields that \eqref{re-K1} can be controlled by
\begin{align*}
&t^{2(\epsilon-\frac n2)}\int_{\{r,r^\prime:|r-r^\prime|<t<r+r^\prime\}}(t^2-(r-r^\prime)^2)^{\frac{n-1}2-\epsilon}\Big(\frac{r^\prime}{r}\Big)^{\frac{n-1}2}\;dr^\prime\\
\lesssim& t^{2(\epsilon-\frac n2)}(t^2)^{\frac{n-1}2-\epsilon}r\int_{\{r,r^\prime:|1-\frac{r^\prime}r|<\frac tr<1+\frac{r^\prime}r\}}\Big(\frac{r^\prime}r\Big)^{\frac{n-1}2}\;d\Big(\frac{r^\prime}r\Big)\\
\lesssim&\Big(\frac{t}r\Big)^{-1}\int_{1-\frac tr}^{1+\frac tr}h^{\frac {n-1}2}\;dh
\lesssim_\epsilon1,
\end{align*}
where we have used $n\geq3$ and $\epsilon\in(\frac12,1)$ again.

\emph{Case b}: $r< t$. In this case, the fact $|r-r^\prime|<t<r+r^\prime$ implies $r^\prime<2t.$ Then, we have
\begin{align*}
&t^{2(\epsilon-\frac n2)}\Big(\frac{t^2-(r-r^\prime)^2}{2rr^\prime}\Big)^{\frac{n-1}2}[t^2-(r-r^\prime)^2]^{-\epsilon}{r^\prime}^{n-1}\\
\lesssim&t^{2(\epsilon-\frac n2)}{r^\prime}^{n-1}\frac{t^2-(r-r^\prime)^2}{2rr^\prime}[(t+|r-r^\prime|)(t-|r-r^\prime|)]^{-\epsilon}\\
\lesssim&t^{2(\epsilon-\frac n2)}t^{n-2}\frac1r t^{1-\epsilon}(t-|r-r^\prime|)^{1-\epsilon}\\
=& \Big(\frac tr\Big)^{\epsilon-1}\frac1r\Big(\frac {t-|r-r^\prime|}r\Big)^{1-\epsilon}\\
\lesssim&\frac1r\Big(\frac {t-|r-r^\prime|}r\Big)^{1-\epsilon}.
\end{align*}
By changing variable $h=\frac{r^\prime}r$, the term in \eqref{re-K1} can be bounded by
\begin{align*}
& \int_{\{r,r^\prime;|r-r^\prime|<t<r+r^\prime\}}\frac1r\Big(\frac {t-|r-r^\prime|}r\Big)^{1-\epsilon}\;dr^\prime\\
\lesssim_{\epsilon}&\int_{\frac tr-1}^{1+\frac tr}\Big(\frac tr-|h-1|\Big)^{1-\epsilon}\;dh
\lesssim_{\epsilon}1,
\end{align*}
since $1-\epsilon>0$, and $\frac tr-|h-1|\leq2$ for all $h\in[\frac tr-1,1+\frac tr].$ Hence, the proof of Lemma \ref{lemK} with $i=1$ follows. \vspace{0.2cm}

{\bf{\emph{The proof for $i=2$}:}} On the one hand, we can acquire
\begin{align*}
&\int_Y\Big[(A-d_Y(y,y^\prime))^{-\epsilon}+\frac1{d_Y(y,y^\prime)}(A-d_Y(y,y^\prime))^{\frac12-\epsilon}\Big]\mathbbm{1}_{\{y,y^\prime;d_Y(y,y^\prime)<A\}}d\mu_Y(y^\prime)\\
\lesssim&\int_0^A\Big[(A-h)^{-\epsilon}+\frac1h(A-h)^{\frac12-\epsilon}\Big]h^{n-2}\;dh\\
\lesssim_\epsilon&1
\end{align*}
for $n\geq3$, $\frac12<\epsilon<1$, and $A\in[\frac\pi2,\pi)$.

On the other hand, for $\frac\pi2\leq A=\arccos\frac{r^2+{r^\prime}^2-t^2}{2rr^\prime}<\pi$, we conclude that
$$\pi-A=2\cdot\frac{\pi-A}2\sim2\cdot\sin\frac{\pi-A}2=2\sqrt{\frac{1-\cos(\pi-A)}2}=2\Big(\frac{(r+r^\prime)^2-t^2}{4rr^\prime}\Big)^{\frac12}.$$
Then, it suffices to show
\begin{align}\label{re-K2}
\int_{\mathcal{N}_{t,r}}t^{2(\epsilon-\frac n2)}(rr^\prime)^{-\epsilon}\Big(\frac{(r+r^\prime)^2-t^2}{4rr^\prime}\Big)^{-\frac12(\frac12+\epsilon)}{r^\prime}^{n-1}\;dr^\prime\lesssim_\epsilon1,
\end{align}
where the notation $\mathcal{N}_{t,r}$ denotes
\begin{align}
\mathcal{N}_{t,r}=\{r^\prime>0:|r-r^\prime|<t<r+r^\prime\quad\text{and}\quad r^2+{r^\prime}^2\leq t^2\}.
\end{align}
The inequality \eqref{re-K2} can be controlled by
\begin{align*}
&\int_{\mathcal{N}_{t,r}}t^{2(\epsilon-\frac n2)}(rr^\prime)^{-\epsilon}\Big(\frac{(r+r^\prime)^2-t^2}{4rr^\prime}\Big)^{-\epsilon}\Big(\frac{(r+r^\prime)^2-t^2}{4rr^\prime}\Big)^{\frac12(-\frac12+\epsilon)}{r^\prime}^{n-1}\;dr^\prime\\
\lesssim_\epsilon&\int_{\mathcal{N}_{t,r}}t^{2(\epsilon-\frac n2)}[(r+r^\prime)^2-t^2]^{-\epsilon}t^{n-1}\;dr^\prime\\
\lesssim_\epsilon&\int_{\mathcal{N}_{t,r}}\Big(\frac tr\Big)^{\epsilon-1}\Big(\frac{r+r^\prime-t}r\Big)^{-\epsilon}\;d\Big(\frac{r^\prime}r\Big)\\
\lesssim_\epsilon&\int_{\{|r^\prime-r|<t<r+r^\prime\}}\Big(\frac{r+r^\prime-t}r\Big)^{-\epsilon}\;d\Big(\frac{r^\prime}r\Big)\\
=_\epsilon&\int_0^2h^{-\epsilon}\;dh
\lesssim_\epsilon 1,
\end{align*}
where we have used the changing variable $h=\frac{r^\prime}r-\Big(\frac tr-1\Big)$ and $\epsilon<1.$

Here we finish the proof of Lemma \ref{lemK} with $i=2$, and so Lemma \ref{lemK} follows.
\end{proof}

\subsection{Proof of Theorem \ref{est-wave}}\label{proof-1}

In this subsection, we illustrate Theorem \ref{est-wave} by utilizing Theorem \ref{theo-Lp} and Stein's interpolation theorem.

Recall the analytic operator family $F_{w,t}(h)$ in \eqref{equ:fanaoper}, and from \cite[p. 463]{T86}, we can see that there exists two constants $C$ and $c$ such that
\begin{equation*}
|F_{w,t}(h)|\leq Ce^{c(\frac{n+1}2-\epsilon)|s|},\quad\forall t,h>0,
\end{equation*}
which implies
\begin{equation}
\|F_{\frac{n+1}2+(\frac{n+1}2-\epsilon)si,t}(\mathcal{L}_A)\|_{L^2(C(Y))\rightarrow L^2(C(Y))}\leq Ce^{c(\frac{n+1}2-\epsilon)|s|}.
\end{equation}
On the other hand, we observe that from \eqref{equ:fanaoper}
$$F_{\frac{n-1}2,t}=\frac{\sin{t\sqrt h}}{t\sqrt h}.$$
Fix $\frac12<\epsilon<1$ and choose $\theta_\epsilon\in(0,1)$ satisfying
\begin{equation*}
\frac{n-1}2=\epsilon+\Big(-\epsilon+\frac{n+1}2\Big)\theta_\epsilon.
\end{equation*}
Also, we set
\begin{equation*}
\frac1{p_\epsilon}=(1-\theta_\epsilon)+\frac{\theta_\epsilon}2=1-\frac{\theta_\epsilon}2\quad\text{or}\quad\frac1{p_\epsilon}=\frac{\theta_\epsilon}2.
\end{equation*}
According to Theorem \ref{theo-Lp} and Stein's interpolating theorem  \cite{S56}, there exists a constant $C_{p_{\epsilon}}>0$ such that
\begin{align*}
\Big\|\frac{\sin t\sqrt{\mathcal{L}_A}}{t\sqrt{\mathcal{L}_A}}\Big\|_{L^{p_\epsilon}(C(Y))\rightarrow L^{p(\epsilon)}(C(Y))}\leq C_{p_{\epsilon}},
\end{align*}
where $p_{\epsilon}$ satisfies
$$\frac1{p_\epsilon}\in\Big(\frac12\frac{n-3}{n-1},\frac12\frac{n-2}2\Big)\bigcup\Big(\frac12+\frac1n,\frac12+\frac1{n-1}\Big)$$
when $\epsilon$ goes through the whole interval $(\frac12,1)$.

This together with Riesz Convex theorem establishes Theorem \ref{est-wave}.

\subsection{Some technical lemmas}\label{proof12}
In this subsection, we focus on the proof of crucial Lemma \ref{lem:kerest1} and Lemma \ref{lem:kerest2} established in previous subsection.

{\bf{\emph{The proof of Lemma \ref{lem:kerest1}:}}} Let $\lambda$ be one eigenvalue of the operator $(i\nabla_y+{\bf A})^2$ on $Y$ such that
$$(i\nabla_y+{\bf A})^2\varphi(y)=\lambda\varphi(y),$$
where $\varphi(y)$ is the corresponding eigenfunction. Since $Y$ is a closed manifold,
the spectral theorem ensures that the set of eigenvalues, denoted by $\{\lambda_j\}_{j=0}^{\infty}$, is discrete. Moreover, these eigenvalues satisfy the ordering $0 = \lambda_0 < \lambda_1 \leq \lambda_2 \leq \cdots$.
Consequently, we have
\begin{align*}
\begin{cases}
(i\nabla_y+{\bf A})^2\varphi_j(y)=\lambda_j\varphi_j(y),\quad\varphi_j(y)\in C^\infty(Y),\\
\int_Y|\varphi_j(y)|^2\;d\mu_Y(y)=1.
\end{cases}
\end{align*}
The set of eigenfunctions $\{\varphi_j\}_{j=0}^\infty$ constitutes an orthonormal basis of the Hilbert space $L^2(Y)$.

Therefore, we can write $H(y,y^\prime)$ as
\begin{equation}
H(y,y^\prime)=\sum_{j=0}^{+\infty}\cos(\pi\sqrt{\lambda_j+\beta^2})e^{-h\sqrt{\lambda_j+\beta^2}}\varphi_j(y)\overline{\varphi_j(y^\prime)}.
\end{equation}
Note that $H(y,y^\prime)=\overline{H(y^\prime,y)}$ and $Y$ is compact, to prove Lemma \ref{lem:kerest1}, it suffices to show
\begin{equation}\label{Hyy}
\sup_{y\in Y}\int_Y|H(y,y^\prime)|^2\;d\mu_Y(y^\prime)\lesssim e^{-2h\beta}(1+h^{\frac{1-n}2})^2.
\end{equation}
Since
\begin{align*}
\int_Y|H(y,y^\prime)|^2\;d\mu_Y(y^\prime)&=\sum_{j=0}^{+\infty}\cos^2(\pi\sqrt{\lambda_j+\beta^2})e^{-2h\sqrt{\lambda_j+\beta^2}}|\varphi_j(y)|^2\\
&\lesssim e^{-\beta h}\sum_{j=0}^{+\infty}e^{-h\sqrt{\lambda_j+\beta^2}}|\varphi_j(y)|^2\\
&=e^{-\beta h}e^{-h\sqrt{(i\nabla_y+{\bf A})^2+\beta^2}}(y,y)\\
&=e^{-\beta h}\frac h{2\sqrt{\pi}}\int_0^{+\infty}u^{-\frac23}e^{-\frac{h^2}{4u}}e^{-\beta^2u}e^{u(i\nabla_y+{\bf A})^2}(y,y)\;du\\
&\lesssim e^{-2h\beta}(1+h^{ {1-n}}),
\end{align*}
which implies \eqref{Hyy}. Thus, we complete the proof of Lemma \ref{lem:kerest1}.\vspace{0.2cm}

{\bf{\emph{The proof of Lemma \ref{lem:kerest2}:}}} To achieve Lemma \ref{lem:kerest2}, we divide into $\mathcal{A}\geq1$ and $\mathcal{A}<1$.

\emph{Case a:} $\mathcal{A}\geq1$. 
Using the inequality \eqref{ine:kere1}, one has
\begin{align}\label{est-T1}
&\int_{\mathcal{A}}^{+\infty}(\|e^{-h\nu}\cos\pi\nu\|_*+\|e^{-h\nu}\cos\pi\nu\|^*)|(\cosh h-\cosh\mathcal{A})^{\beta-w}|\;dh\nonumber\\
\lesssim& \int_{\mathcal{A}}^{+\infty} e^{-h\beta}(\cosh h-\cosh\mathcal{A})^{\beta-\epsilon}\;dh\nonumber\\
=&\Big(\frac \pi2\Big)^{-\frac12}e^{-i\pi(\epsilon-\frac {n-1}2)}\Gamma\Big(\frac{n}2-\epsilon\Big)(\sinh \mathcal{A})^{\frac{n-1}2-\epsilon}Q_{\beta-\frac12}^{\epsilon-\frac{n-1}2}(\cosh\mathcal{A}),
\end{align}
where we have employed \eqref{gammaQ} for the integral transformation.
%

On the other hand, from \cite[p.197]{MOS66}, the following asymptotic formula for the Legendre function $Q_a^b(z)$ holds
\begin{equation*}
Q_a^b(z)=e^{i\pi b}2^{-a-1}\sqrt{\pi}\frac{\Gamma(a+b+1)}{\Gamma(\frac32+a)}z^{-a-1}(1+o(1)),\quad\forall|z|\gg1.
\end{equation*}
Therefore, it follows that
\begin{align}
\eqref{est-T1}
\lesssim_{\epsilon} (\sinh \mathcal{A})^{\frac{n-1}2-\epsilon}(\cosh \mathcal{A})^{-\frac{n-1}2}\lesssim_{\epsilon} (\cosh\mathcal{A}-1)^{-\epsilon}
\end{align}
for $\frac12<\epsilon<1$, $n\geq3$ and $\mathcal{A}\geq1$.\vspace{0.2cm}

\emph{Case b:} $0<\mathcal{A}<1$. We decompose
\begin{align}\label{ine:A1}
&\int_{\mathcal{A}}^{+\infty}(\|e^{-h\nu}\cos\pi\nu\|_*+\|e^{-h\nu}\cos\pi\nu\|^*)|(\cosh h-\cosh\mathcal{A})^{\beta-w}|\;dh\nonumber\\
=&\Big\{\int_{\mathcal{A}}^{2\mathcal{A}}+\int_{2\mathcal{A}}^{1+\mathcal{A}}+\int_{1+\mathcal{A}}^{+\infty}\Big\}
(\|e^{-h\nu}\cos\pi\nu\|_*+\|e^{-h\nu}\cos\pi\nu\|^*)|(\cosh h-\cosh\mathcal{A})^{\beta-w}|\;dh\nonumber\\
:=&H_1+H_2+H_3.
\end{align}
\emph{The estimation of $H_3$:} For the term $H_3$, we first notice that
\begin{equation}\label{e-cosh}
\cosh a-\cosh b=2\sinh\frac{a+b}2\sinh\frac{a-b}2.
\end{equation}
Then, we use the variable change $h=s+\mathcal{A}$ to achieve
\begin{align*}
H_3&\lesssim\int_{1+\mathcal{A}}^\infty e^{-h\beta}(\cosh h-\cosh\mathcal{A})^{\beta-\epsilon}\;dh\\
&=e^{-\beta \mathcal{A}}\int_1^{+\infty}e^{-\beta s}(\cosh (\mathcal{A}+s)-\cosh\mathcal{A})^{\beta-\epsilon}\;ds\\
&=e^{-\beta \mathcal{A}}\int_1^{+\infty}e^{-\beta s}\Big(2\sinh\frac{s+2\mathcal{A}}2\sinh\frac{s}2\Big)^{\beta-\epsilon}\;ds\\
&\lesssim_{\epsilon}\int_1^{+\infty}e^{-\beta s}(e^s)^{\beta-\epsilon}\;ds
\lesssim_{\epsilon}\frac1\epsilon,
\end{align*}
because of the fact
$\sinh s\sim e^s$ when $s\geq1$.

\emph{The estimation of $H_2$ and $H_1$:} For $0<\mathcal{A}<1$, we observe that
\begin{align*}
\mathcal{A}^2\sim4\sinh^2\frac{\mathcal{A}}2=2(\cosh\mathcal{A}-1).
\end{align*}
Hence, by \eqref{ine:kere1} and \eqref{e-cosh}, we can acquire
\begin{align*}
H_2&\lesssim\int_{2\mathcal{A}}^{1+\mathcal{A}}h^{1-n}\Big(2\sinh\frac{h+\mathcal{A}}2\sinh\frac{h-\mathcal{A}}2\Big)^{\beta-\epsilon}\;dh\\
&\lesssim\int_{2\mathcal{A}}^{1+\mathcal{A}}h^{1-n}h^{2\beta-2\epsilon}\;dh\\
&\lesssim_{\epsilon}\frac{1}{\epsilon}\mathcal{A}^{-2\epsilon}
\lesssim_{\epsilon} (\cosh\mathcal{A}-1)^{-\epsilon}.
\end{align*}
In the second inequality we utilized the fact
$$\sinh \frac{h+\mathcal{A}}2\sim\frac{h+\mathcal{A}}2\sim h,\ \  \sinh \frac{h-\mathcal{A}}2\sim\frac{h-\mathcal{A}}2\sim h,\ \ \text{for}\ \ 2\mathcal{A}\leq h\leq1+\mathcal{A},$$ and in the last inequality, we used $\frac12<\epsilon<1$, $n\geq3$ and $\beta=(n-2)/2$.

Next, we aim to estimate the term $H_1$. First notice that
$$\sinh \frac{h+\mathcal{A}}2\sim\frac{h+\mathcal{A}}2\sim \mathcal{A},\ \ \sinh \frac{h-\mathcal{A}}2\sim\frac{h-\mathcal{A}}2,\ \text{for}\ \mathcal{A}\leq h\leq2\mathcal{A}.$$
Using the similar argument as above, we conclude that
\begin{align*}
H_1\lesssim&\int_{\mathcal{A}}^{2\mathcal{A}}h^{1-n}\Big(2\sinh\frac{h+\mathcal{A}}2\sinh\frac{h-\mathcal{A}}2\Big)^{\beta-\epsilon}\;dh\\
\lesssim&\mathcal{A}^{1-n}\mathcal{A}^{\beta-\epsilon}\int_{\mathcal{A}}^{2\mathcal{A}}{(h-\mathcal{A})}^{\beta-\epsilon}\;dh\\
=&\frac{1}{\frac n2-\epsilon}\mathcal{A}^{-2\epsilon}
\lesssim_{\epsilon} (\cosh\mathcal{A}-1)^{-\epsilon},
\end{align*}
So far, we have finished the proof of Lemma \ref{lem:kerest2}.

\section{The proof of Proposition \ref{lemT}}\label{proof}
Before proceeding with the proof of Proposition \ref{lemT}, we first give some key results related to the Hadamard parametrix of $\cos (s\nu)$ with $\nu=\sqrt{(i\nabla_y+{\bf A})^2+\alpha^2}$, see the Appendix for details.

\begin{lemma}\label{hp}
Suppose that $Y$ is a compact connected Riemannian manifold without boundary for dimension $n-1$, and its injectivity radius $\rho>\pi$. Then, for $K\in\mathbb{N}^*$, $0<h\leq\pi$, and $s>0$, we have
\begin{equation}\label{def:cos}
\cos (s\nu)(y,y^\prime)=\sum_{k=0}^KU_k(y,y^\prime)s\frac{(s^2-d^2_Y(y,y^\prime))^{k-\frac n2}_+}{\Gamma(k-\frac n2+1)}+sC_K(s,y,y^\prime),
\end{equation}
where

1. $U_k(y,y^\prime)\in C^\infty(Y\times Y)$.

2. $C_K(s,y,y^\prime)\in C([0,\pi]\times Y\times Y)$. Moreover, for $K$ large enough, for example $K=n+1$, $C_K(s,y,y^\prime)=0$ when $d_Y(y,y^\prime)>s$.\\

3. For $g\in C^\infty(Y)$ and $y\in Y$, there holds
\begin{equation}
\overline{g_y}(h)=\int_{S^h(y)}g(y^\prime)\;d\sigma(y^\prime),\quad 0<h<\rho,
\end{equation}
where $S^h(y)=\{y^\prime\in Y;d_Y\{y,y^\prime\}=h\}$ and $d\sigma(y^\prime)$ denotes the induced surface measure. Then, for $\frac n2-k\geq1$, the direction of the distribution $\frac{(s^2-d^2_Y(y,y^\prime))^{k-\frac n2}_+}{\Gamma(k-\frac n2+1)}$ (for $0<s<\rho$) is as follows
\begin{align}
\left\langle\frac{(s^2-d^2_Y(y,y^\prime))^{k-\frac n2}_+}{\Gamma(k-\frac n2+1)},g\right\rangle(y)=\left\langle\frac{(s^2-h^2)^{k-\frac n2}_+}{\Gamma(k-\frac n2+1)},\overline{g_y}(h)\right\rangle\nonumber\\=
\begin{cases}\label{eo}
(\frac1{2s}\frac\partial{\partial s})^{\frac n2-k-1}[\frac1{2s}\overline{g_y}(s)],&\text{for}\ \ \frac n2\in\mathbb{N},\\
(\frac1{2s}\frac\partial{\partial s})^{\frac n2-k-1}\int_0^s\frac{(s^2-h^2)^{-\frac12}}{\Gamma(\frac12)}\overline{g_y}(h)\;dh,&\text{for}\ \ \frac n2 \notin\mathbb{N}.
\end{cases}
\end{align}
\end{lemma}
We are now in position to utilize Lemma \ref{hp} to show Proposition \ref{lemT}.\vspace{0.2cm}

{\bf{\emph{The proof of Proposition \ref{lemT}:}}}  For $0< A<\pi$, we write
\begin{equation}
G:=\int_0^A(\cos h-\cos A)^{\beta-w}\cos (h\nu) (y,y^\prime)\;dh .
\end{equation}
Hence, for the sake of demonstrating Proposition \ref{lemT}, we are reduced to show
\begin{align}\label{est:I}
|G|\lesssim_\epsilon & \mathbbm{1}_{\{y,y^\prime: d_Y(y,y^\prime)<A\}}\\
&\times
\begin{cases}
(A^2-d_{Y}^2(y,y^\prime))^{-\epsilon}\Big[1+\frac{(A^2-d_{Y}^2(y,y^\prime))^{\frac12}}{d_{Y}(y,y^\prime)}\Big],\quad0<A<\frac\pi2,\nonumber\\
(A-d_{Y}(y,y^\prime))^{-\epsilon}(\pi-A)^{-\frac12-\epsilon}\Big[1+\frac{(\pi-A)(A-d_{Y}(y,y^\prime))^{\frac12}}{d_{Y}(y,y^\prime)}\Big],\quad\frac\pi2\leq A<\pi.\nonumber
\end{cases}
\end{align}
Denote
\begin{align*}
G_{e}=\int_0^A(\cos h-\cos A)^{\beta-w}hC_{n+1}(h,y,y^\prime)\;dh,
\end{align*}
and
\begin{align*}
G_j=\int_0^A(\cos h-\cos A)^{\beta-w}h\frac{(h^2-d^2_Y(y,y^\prime))^{j-\frac n2}_+}{\Gamma(j-\frac n2+1)}\;dh,\qquad \text{for}\ 0\leq j\leq n+1.
\end{align*}
According to \eqref{def:cos}, it further suffices to prove that \eqref{est:I} holds for $G_{e}$ and $G_j(0\leq j\leq n+1).$

Since $\frac12<\epsilon<1$, $n\geq3$, $0< A<\pi$, and the fact $C_{n+1}(h,y,y^\prime)=0$ when $d_Y(y,y^\prime)>h$, one has
\begin{align}\label{est-rest}
|G_{e}|
&\lesssim\mathbbm{1}_{\{y,y^\prime:d_Y(y,y^\prime)<A\}}\int_{d_Y(y,y^\prime)}^A(\cos h-\cos A)^{\frac{n-2}2-\epsilon}h\;dh\nonumber\\
&\lesssim_\epsilon\mathbbm{1}_{\{y,y^\prime:d_Y(y,y^\prime)<A\}},
\end{align}
where we have used the fact that
for $\alpha<-\frac12$, there exists a constant $C_\alpha>0$ satisfying
$$\int_a^b(\cos h-\cos b)^{-\alpha}h\;dh\leq C_\alpha,\quad\forall\ 0\leq a\leq b\leq\pi.$$
This means \eqref{est:I} holds for $G_{e}$.\vspace{0.2cm}

We now turn to consider \eqref{est:I} for $G_j$ $(0\leq j\leq n+1)$.
For this purpose, we divide into two cases $\frac n2\in\mathbb{N}$ and $\frac n2\notin\mathbb{N}$.\vspace{0.1cm}

{\bf{\emph{Case a}:}}
$\frac n2\in\mathbb{N}$. We first observe the following two inequalities:
\begin{equation}\label{cos-1}
\cos a-\cos b\sim b^2-a^2,\quad\text{if}\quad0\leq a\leq b<\frac\pi2,
\end{equation}
and
\begin{equation}\label{cos-2}
\cos a-\cos b\geq C(b-a)(\pi-b),\quad\text{for all}\quad0\leq a\leq b\quad\text{with}\quad\frac\pi2\leq b<\pi.
\end{equation}
For $j=0$, since the even function $\frac{\sin h}h$ is analytical for $-\pi<h<\pi$, we use \eqref{eo} to attain
\begin{align}\label{even-1}
|G_0|&=\Big|\mathbbm{1}_{\{y,y^\prime:d_Y(y,y^\prime)<A\}}\frac12\Big(-\frac1{2h}\frac d{dh}\Big)^{\frac n2-1}(\cos h-\cos A)^{\frac n2-1-\omega}|_{h=d_Y(y,y^\prime)}\Big|\nonumber\\
&\lesssim_\epsilon\mathbbm{1}_{\{y,y^\prime:d_Y(y,y^\prime)<A\}}(\cos d_Y(y,y^\prime)-\cos A)^{-\epsilon}.
\end{align}
For $1\leq j\leq n+1$, a similar argument yields
\begin{equation}
|G_j|\lesssim_\epsilon\mathbbm{1}_{\{y,y^\prime:d_Y(y,y^\prime)<A\}}.
\end{equation}
Applying \eqref{cos-1} and \eqref{cos-2} to \eqref{even-1}, we then obtain the desired result for $\frac n2\in\mathbb{N}$.\vspace{0.2cm}

{\bf{\emph{Case b}:}} $\frac n2\notin\mathbb{N}$, that is $n=2k+1$ with $ k\in\mathbb{N}$ and $ k\geq1$. In this case, the difficulty comes from estimating $G_0$ because of the function $(\cos h-\cos A)^{\alpha-w}\in C^{k-1}$ but $\notin C^{k}$. However, by \eqref{eo}, we can acquire
\begin{align*}
G_0=&\mathbbm{1}_{\{y,y^\prime:d_Y(y,y^\prime)<A\}}\\
&\times\int_{d_Y(y,y^\prime)}^A\Big[(-\frac1{2h}\frac{d}{dh})^{k-1}(\cos h-\cos A)^{k-\frac12-\omega}\Big]h\frac{(h^2-d_Y^2(y,y^\prime))^{-\frac32}_+}{\Gamma(-\frac12)}\;dh\\
=&\mathbbm{1}_{\{y,y^\prime:d_Y(y,y^\prime)<A\}}(G_{0,1}+G_{0,2}),
\end{align*}
where
\begin{align*}
G_{0,1}=&\int_{\frac{A+d_{Y}(y,y^\prime)}2}^A\Big[\Big(-\frac1{2h}\frac d{dh}\Big)^{k-1}(\cos h-\cos A)^{k-\frac12-w}\Big]h\frac{(h^2-d^2_Y(y,y^\prime))^{-\frac 32}_+}{\Gamma(-\frac12)}\;dh,\\
G_{0,2}=&\int_{d_{Y}(y,y^\prime)}^{\frac{A+d_{Y}(y,y^\prime)}2}
\Big[\Big(-\frac1{2h}\frac d{dh}\Big)^{k-1}(\cos h-\cos A)^{k-\frac12-w}\Big]h\frac{(h^2-d^2_Y(y,y^\prime))^{-\frac 32}_+}{\Gamma(-\frac12)}\;dh.
\end{align*}\vspace{0.2cm}
{\bf{\emph{Estimation for $|G_{0,1}|$}:}} Since $0\leq d_{Y}(y,y^\prime)<A<\pi$, we have
\begin{align*}
&\Big|\int_{\frac{A+d_{Y}(y,y')}2}^A\Big[(-\frac1{2h}\frac d{dh})^{k-1}(\cos h-\cos A)^{k-\frac12-w}\Big]h\frac{(h^2-d^2_Y(y,y^\prime))^{-\frac 32}_+}{\Gamma(-\frac12)}\;dh\Big|\\
\lesssim_\epsilon&\int_{\frac{A+d_{Y}(y,y^\prime)}2}^A\Big|(\cos h-\cos A)^{\frac12-\epsilon}h(h^2-d^2_Y(y,y^\prime))^{-\frac 32}_+\Big|\;dh\\
\lesssim_\epsilon&(A+d_Y(y,y^\prime))(A^2-d_Y^2(y,y^\prime))^{-\frac32}\int_{\frac{A+d_{Y}(y,y^\prime)}2}^A\Big(2\sin\frac{A-h}2\sin\frac{A+h}2\Big)^{\frac12-\epsilon}\;dh\\
\lesssim_\epsilon&(A+d_Y(y,y^\prime))^{-\frac12}(A-d_Y(y,y^\prime))^{-\frac32}\int_{\frac{A+d_{Y}(y,y^\prime)}2}^A\Big((A-h)\sin\frac{A+h}2\Big)^{\frac12-\epsilon}\;dh.
\end{align*}
Next, we divide $(0,\pi)$ into two intervals $(0,\frac\pi2)\cup[\frac\pi2,\pi)$.

\emph{Case 1}: $0<A<\frac\pi2$. In this case, we obtain
\begin{align*}
&\int_{\frac{A+d_{Y}(y,y^\prime)}2}^A\Big((A-h)\sin\frac{A+h}2\Big)^{\frac12-\epsilon}\;dh\\
\lesssim_\epsilon&\int_{\frac{A+d_{Y}(y,y^\prime)}2}^A\Big((A-h)(A+h)\Big)^{\frac12-\epsilon}\;dh\\
\lesssim_\epsilon&(A+d_Y(y,y^\prime))^{\frac12-\epsilon}\int_{\frac{A+d_{Y}(y,y^\prime)}2}^A(A-h)^{\frac12-\epsilon}\;dh\\
\lesssim_\epsilon&(A+d_Y(y,y^\prime))^{\frac12-\epsilon}(A-d_Y(y,y^\prime))^{\frac32-\epsilon}.
\end{align*}

\emph{Case 2}: $\frac\pi2\leq A<\pi$. In this case, one has
\begin{align*}
&\int_{\frac{A+d_{Y}(y,y^\prime)}2}^A\Big((A-h)\sin\frac{A+h}2\Big)^{\frac12-\epsilon}\;dh\\
\lesssim_\epsilon&\int_{\frac{A+d_{Y}(y,y^\prime)}2}^A\Big((A-h)(\pi-\frac{A+h}2)\Big)^{\frac12-\epsilon}\;dh\\
\lesssim_\epsilon&(\pi-A)^{\frac12-\epsilon}\int_{\frac{A+d_{Y}(y,y^\prime)}2}^A(A-h)^{\frac12-\epsilon}\;dh\\
\lesssim_\epsilon&(\pi-A)^{\frac12-\epsilon}(A-d_Y(y,y^\prime))^{\frac32-\epsilon}.
\end{align*}
Combining with the above inequalities, we can prove \eqref{est:I} holds for $G_{0,1}$.
\vspace{0.2cm}

{\bf{\emph{Estimation for $|G_{0,2}|$}:}} For any $0<h<\rho$, there holds
$$(h^2-d^2_Y(y,y^\prime))^{-\frac 32}_+=(h+d_Y(y,y^\prime))^{-\frac 32}(h-d_Y(y,y^\prime))^{-\frac 32}_+,$$
see \cite[Section 3.2]{GS}. In the following, we continue to take account of the following Riemann-Liouville integral
\begin{align*}
\frac1{\Gamma(-\frac12)}\int_{d_{Y}(y,y^\prime)}^{\frac{A+d_{Y}(y,y^\prime)}2}h(h+d_Y(y,y^\prime))^{-\frac 32}
\Big[\Big(-\frac1{2h}\frac d{dh}\Big)^{k-1}(\cos h-\cos A)^{k-\frac12-w}\Big](h-d_Y(y,y^\prime))^{-\frac 32}_+\;dh,
\end{align*}
for $0<d_{Y}(y,y^\prime)<A<\pi.$ Utilizing the result in \cite[p.10]{R49}, we derive
\begin{align}\label{dec-1}
G_{0,2}
=&\frac1{\Gamma(-\frac12)}f\Big(\frac{A+d_Y(y,y^\prime)}2\Big)\Big(\frac{A+d_Y(y,y^\prime)}2-d_Y(y,y^\prime)\Big)_+^{-\frac 12}\;dh\\ \label{dec-2}
&+\frac1{\Gamma(-\frac12)}\int_{d_Y(y,y^\prime)}^{\frac{A+d_Y(y,y^\prime)}2}f^\prime(h)(h-d_Y(y,y^\prime))_+^{-\frac 12}\;dh,
\end{align}
where we write
\begin{equation}\label{def:fh}
f(h):=h(h+d_Y(y,y^\prime))^{-\frac32}\Big(-\frac1{2h}\frac d{dh}\Big)^{k-1}(\cos h-\cos A)^{k-\frac12-w}.
\end{equation}

\emph{Estimation for \eqref{dec-1}:} For $f\Big(\frac{A+d_{Y}(y,y^\prime)}2\Big)$, a simple computation yields
\begin{align*}
&f\Big(\frac{A+d_{Y}(y,y^\prime)}2\Big)\\
\lesssim_\epsilon&(A+d_{Y}(y,y^\prime))^{-\frac12}\Big(\cos \frac{A+d_{Y}(y,y^\prime)}2-\cos A\Big)^{\frac12-\epsilon}\\
\lesssim_\epsilon&(A+d_{Y}(y,y^\prime))^{-\frac12}\Big(\sin\frac{\frac{A+d_{Y}(y,y^\prime)}2+A}2\sin\frac{A-\frac{A+d_{Y}(y,y^\prime)}2}2\Big)^{\frac12-\epsilon}\\
\lesssim_\epsilon&(A+d_{Y}(y,y^\prime))^{-\frac12}\times
\begin{cases}
(A^2-d_Y^2(y,y^\prime))^{\frac12-\epsilon},&0<A<\frac\pi2,\\
(A-d_Y(y,y^\prime))^{\frac12-\epsilon}(\pi-A)^{\frac12-\epsilon},&\frac\pi2\leq A<\pi,
\end{cases}
\end{align*}
which implies \eqref{dec-1} can be controlled by the right side of \eqref{est:I}.\vspace{0.1cm}

\emph{Estimation for \eqref{dec-2}:} We now focus on estimating the more complicated part
$$\frac{1}{\Gamma(-\frac12)}\int_{d_{Y}(y,y^\prime)}^{\frac{A+d_{Y}(y,y^\prime)}2}f^\prime(h)(h-d^2_Y(y,y^\prime))^{-\frac 12}_+\;dh.$$
For $0<d_Y(y,y^\prime)\leq h\leq\frac{A+d_Y(y,y^\prime)}2$, we first calculate
\begin{align*}
f^\prime(h)=&(h+d_Y(y,y^\prime))^{-\frac32}\Big(-\frac1{2h}\frac d{dh}\Big)^{k-1}(\cos h-\cos A)^{k-\frac12-w}\\
&-\frac32 h(h+d_Y(y,y^\prime))^{-\frac52}\Big(-\frac1{2h}\frac d{dh}\Big)^{k-1}(\cos h-\cos A)^{k-\frac12-w}\\
&+h(h+d_Y(y,y^\prime))^{-\frac32}\frac d{dh}\Big[\left(-\frac1{2h}\frac d{dh}\right)^{k-1}(\cos h-\cos A)^{k-\frac12-w}\Big].
\end{align*}
Note that  for $0<d_Y(y,y^\prime)\leq h\leq\frac{A+d_Y(y,y^\prime)}2<A<\pi$ and $\Re w=\epsilon\in(\frac12,1)$, one has
\begin{align*}
&\Big|\Big(-\frac1{2h}\frac d{dh}\Big)^{k-1}(\cos h-\cos A)^{k-\frac12-w}\Big|\lesssim_{\epsilon}(\cos h-\cos A)^{\frac12-\epsilon},
\end{align*}
and
\begin{align*}
\Big|\frac d{dh}\Big[\left(-\frac1{2h}\frac d{dh}\right)^{k-1}(\cos h-\cos A)^{k-\frac12-w}\Big]\Big|
\lesssim\epsilon\Big[(\cos h-\cos A)^{-\frac12-\epsilon}h+(\cos h-\cos A)^{\frac12-\epsilon}\Big],
\end{align*}
and it follows that
\begin{equation*}
|f^\prime(h)|\lesssim_{\epsilon}\Big[h^{\frac12}(\cos h-\cos A)^{-\frac12-\epsilon}+h^{-\frac 32}(\cos h-\cos A)^{\frac12-\epsilon}\Big].
\end{equation*}
Furthermore, the term in \eqref{dec-2} can be controlled by
\begin{align*}
&\Big|\frac1{\Gamma(-\frac12)}\int_{d_Y(y,y^\prime)}^{\frac{A+d_Y(y,y^\prime)}2}f^\prime(h)(h-d_Y(y,y^\prime))_+^{-\frac 12}\;dh\Big|\\
\lesssim_\epsilon&\int_{d_Y(y,y^\prime)}^{\frac{A+d_Y(y,y^\prime)}2}\Big[h^{\frac12}(\cos h-\cos A)^{-\frac12-\epsilon}\Big](h-d_Y(y,y^\prime))^{-\frac12}_{+}\;dh\\
&+\int_{d_Y(y,y^\prime)}^{\frac{A+d_Y(y,y^\prime)}2}\Big[h^{-\frac 32}(\cos h-\cos A)^{\frac12-\epsilon}\Big](h-d_Y(y,y^\prime))^{-\frac12}_{+}\;dh.
\end{align*}
And then using the identity
$$\cos a-\cos b=2\sin \frac{b-a}2\sin\frac{a+b}2,$$
the term in \eqref{dec-2} can be further dominated by
\begin{align*}
(A+d_Y(y,y^\prime))^{\frac12}(A-d_Y(y,y^\prime))^{-\frac12-\epsilon}\int_{d_Y(y,y^\prime)}^{\frac{A+d_Y(y,y^\prime)}2}\Big(\sin \frac{A+h}2\Big)^{-\frac12-\epsilon}(h-d_Y(y,y^\prime))^{-\frac12}\;dh,
\end{align*}
and
\begin{align*}
(A-d_Y(y,y^\prime))^{\frac12-\epsilon}\int_{d_Y(y,y^\prime)}^{\frac{A+d_Y(y,y^\prime)}2}h^{-\frac32}\Big(\sin \frac{A+h}2\Big)^{\frac12-\epsilon}(h-d_Y(y,y^\prime))^{-\frac12}\;dh.
\end{align*}
To get our desired result, we will now consider two cases: $0<A<\frac\pi2$ and $\frac\pi2\leq A<\pi.$\vspace{0.1cm}

\emph{Case 1:} $0<A<\frac\pi2$. In this case, we first acquire
\begin{align*}
&\int_{d_Y(y,y^\prime)}^{\frac{A+d_Y(y,y^\prime)}2}\Big(\sin \frac{A+h}2\Big)^{-\frac12-\epsilon}(h-d_Y(y,y^\prime))^{-\frac12}\;dh\\
\lesssim_\epsilon&(A+d_Y(y,y^\prime))^{-\frac12-\epsilon}\int_{d_Y(y,y^\prime)}^{\frac{A+d_Y(y,y^\prime)}2}(h-d_Y(y,y^\prime))^{-\frac12}\;dh\\
\lesssim_\epsilon&(A+d_Y(y,y^\prime))^{-\frac12-\epsilon}(A-d_Y(y,y^\prime))^{\frac12}.
\end{align*}
Next, using variable change $s=h-d_Y(y,y^\prime)$ and $s=hd_Y(y,y^\prime)$, we can derive
\begin{align*}
&\int_{d_Y(y,y^\prime)}^{\frac{A+d_Y(y,y^\prime)}2}h^{-\frac32}\Big(\sin \frac{A+h}2\Big)^{\frac12-\epsilon}(h-d_Y(y,y^\prime))^{-\frac12}\;dh\\
\lesssim_\epsilon&(A+d_Y(y,y^\prime))^{\frac12-\epsilon}\int_{d_Y(y,y^\prime)}^{\frac{A+d_Y(y,y^\prime)}2}h^{-\frac32}(h-d_Y(y,y^\prime))^{-\frac12}\;dh\\
\lesssim_\epsilon&(A+d_Y(y,y^\prime))^{\frac12-\epsilon}\int_{0}^{\frac{A-d_Y(y,y^\prime)}2}{(s+d_Y(y,y^\prime)}^{-\frac32}s^{-\frac12}\;ds\\
\lesssim_\epsilon&(A+d_Y(y,y^\prime))^{\frac12-\epsilon}\frac1{d_Y(y,y^\prime)}\int_{0}^{\frac{A-d_Y(y,y^\prime)}{2d_Y(y,y^\prime)}}{(h+1)}^{-\frac32}h^{-\frac12}\;dh\\
\lesssim_\epsilon&(A+d_Y(y,y^\prime))^{\frac12-\epsilon}\frac1{d_Y(y,y^\prime)}.
\end{align*}

\emph{Case 2:} $\frac\pi2\leq A<\pi.$ In this case, a simple computation yields
\begin{align*}
&\int_{d_Y(y,y^\prime)}^{\frac{A+d_Y(y,y^\prime)}2}(\sin \frac{A+h}2)^{-\frac12-\epsilon}(h-d_Y(y,y^\prime))^{-\frac12}\;dh\\
\lesssim_\epsilon&(\pi-A)^{-\frac12-\epsilon}\int_{d_Y(y,y^\prime)}^{\frac{A+d_Y(y,y^\prime)}2}(h-d_Y(y,y^\prime))^{-\frac12}\;dh\\
\lesssim_\epsilon&(\pi-A)^{-\frac12-\epsilon}(A-d_Y(y,y^\prime))^{\frac12}.
\end{align*}
Then, utilizing a similar argument as above, we gain
\begin{align*}
&\int_{d_Y(y,y^\prime)}^{\frac{A+d_Y(y,y^\prime)}2}h^{-\frac32}(\sin \frac{A+h}2)^{\frac12-\epsilon}(h-d_Y(y,y^\prime))^{-\frac12}\;dh\\
\lesssim_\epsilon&(\pi-A)^{\frac12-\epsilon}\int_{d_Y(y,y^\prime)}^{\frac{A+d_Y(y,y^\prime)}2}h^{-\frac32}(h-d_Y(y,y^\prime))^{-\frac12}\;dh\\
\lesssim_\epsilon&(\pi-A)^{\frac12-\epsilon}\frac1{d_Y(y,y^\prime)},
\end{align*}
which indicates \eqref{dec-2} can be bounded by \eqref{est:I}, so the desired result holds for $\frac n2\notin\mathbb{N}.$\vspace{0.1cm}

For $1\leq j\leq n+1$, by \eqref{eo}, there holds
\begin{align*}
G_j=&\mathbbm{1}_{\{y,y^\prime:d_Y(y,y^\prime)<A\}}\mathbbm{1}_{\{r,r^\prime>0:|r-r^\prime|<t<r+r^\prime\}}\\
&\times\int_{d_Y(y,y^\prime)}^A\Big[(-\frac1{2h}\frac{d}{dh})^{k-1}(\cos h-\cos A)^{k-\frac12-\omega}\Big]h\frac{(h^2-d_Y^2(y,y^\prime))^{j-\frac32}_+}{\Gamma(j-\frac12)}\;dh.
\end{align*}
Using the analogous argument as above, it is easy to check that
\begin{align*}
|G_j|\lesssim_\epsilon\mathbbm{1}_{\{y,y^\prime:d_Y(y,y^\prime)<A\}}.
\end{align*}
It is clear that $G_j$ can be bounded by \eqref{est:I}. Therefore, this completes the proof of Lemma \ref{lemT}.

\section{Appendix}

In this section, we will construct the Hadamard parametrix for $\cos(s\nu)$. First, we recall some definitions.
Let $\square:=\pa_{t}^2-\Delta$ and define distributions as follows
$$W_{a}(t,x)=\lim_{\epsilon\to0^+}{\rm Im}(|x|^2-(t+i\epsilon)^2)^a,\ \ a\in-\frac12\mathbb{N}.$$
In fact, we can write
$$W_a(t,x)=C_a\frac{(t^2-|x|^2)_{+}^a}{\Gamma(a+1)}.$$
Based on the definition of $W_a(t,x)$, it is natural to set
$$W_a(t,x)=H(t^2-|x|^2)(t^2-|x|^2)^a,\ a\in \frac12\mathbb{N},$$
and
$$W_0(t,x)=H(t^2-|x|^2)$$
with $H$ being the Heaviside function. Next, we choose $c_\nu$ and suppose
$$E_\nu=\alpha_\nu W_{-(n-1)/2+\nu},\ \nu=1,2,3,...,$$
such that $\square E_\nu=\nu E_{\nu-1}$. In particular, we have
$$E_{\nu}(t,x)=C_\nu\lim_{\epsilon \to 0^+}(2\pi)^{-n-1}\iint e^{ix\cdot\xi+it\tau}(|\xi|^2-(\tau-i\epsilon)^2)^{-\nu-1}d\xi d\tau$$
in the sense of distribution by the Paley--Weiner theorem (see \cite{Hor,Rudin}). Then we can get the following proposition, see \cite{Sogge} for more details.
\begin{proposition}\label{funcdamential}
Let $E_{\nu}, \nu=0,1,2,...$ be the distributions defined as above. Then
$$supp\ E_v\subset\{(t,x)\in [0,\infty)\times\R^n:|x|\leq t\}$$
$$\square E_0=\delta_{0,0},$$
and for $\nu=1,2,3,...$
$$\square E_\nu=\nu E_{\nu-1}, -2\nabla E_\nu=xE_{\nu-1}, 2\pa_t E_\nu=tE_{\nu-1}.$$
Furthermore, we have
$$E_\nu(0^+)=\lim_{t\to0^+}E(t,x)=0, \nu=0,1,2,...$$
and $$\pa_t^kE_\nu(0^+)=0\ for\ k\leq 2\nu.$$
\end{proposition}
\begin{remark}
  Since $$E_0=c_0(2\pi)^{-n}\int_{\R^n}e^{ix\cdot\xi}\frac{\sin t|\xi|}{|\xi|}d\xi,$$
  using the property of $2\pa_t E_\nu=tE_{\nu-1}$ and by an induction, one can shows that $E_\nu(\nu\geq1)$ can be written by
 \begin{align*}
   E_\nu(t,x)=\sum_{j+k=\nu}H(t)t^j(2\pi)^{-n}&\Big(A_{j,k}\int_{\R^n}e^{ix\cdot\xi}\sin t|\xi||\xi|^{-\nu-1-k}d\xi\\
   &+B_{j,k}\int_{\R^n}e^{ix\cdot\xi}\cos t|\xi||\xi|^{-\nu-1-k}d\xi\Big).
 \end{align*}
\end{remark}

Next, we record some results about the geodesics and normal coordinates.
Let $M$ be smooth Riemannian manifold and $g$ denotes the Riemannian metric, i.e.,$(M, g)$.
\begin{theorem}[\cite{Sogge}]\label{g-coordination}
  If $\Omega$ is a fixed relatively compact subset of $M$, then for $z\in \Omega$ there are coordinates $\kappa_{z}(y)=x$ vanishing at $z$ satisfying $\kappa_z'(z)=I_n$ and
  \begin{align}\label{coordination}
\sum_jg_{ij}(x)x_j=\sum_jg_{ij}(0)x_j,\ \ i=1,...,n,
  \end{align}
  and
  $$(d_g(0,x))^2=\sum_{i,j}g_{ij}(0)x_ix_j.$$
  One can chooses $\delta>0$ such that for every $z\in Y$, $\kappa_z$ is a diffeomorphism from all points $y\in M$ with $d_g(z,y)<\delta$ to points $x$ satisfying $\sum_{i,j}g_{ij}(0)x_ix_j<\delta^2$.
\end{theorem}
\begin{remark}
  By a linear change of variables, one can assume $(g_{ij})(0)=I_n$.
\end{remark}

We now construct Hadamard parametrix for constant coefficient differential operators. Let $T:\R^n\to\R^n$ be a linear bijection such that $(g_{ij})=T^tT$. Thus the operator $\Delta=\pa_{y_1}^2+\cdot\cdot\cdot+\pa_{y_n}^2$ can be written via $y=Tx$
$$\sum_{j,k=1}^n\frac{\pa}{\pa x_j}g^{jk}\frac{\pa}{\pa x_k}.$$
Therefore, it follows that
\begin{align}\label{Cons}
  \left(\pa_t^2-\sum_{j,k=1}^n\frac{\pa}{\pa x_j}g^{jk}\frac{\pa}{\pa x_k}\right)E_0(t,|x|_g)=|g|^{-\frac12}\delta_{0,0}(t,x),
\end{align}
with $|x|_g^2=\sum_{i,k}g_{jk}x_jx_k$.

Next, we will employ \eqref{coordination} to convert variable coefficients to constant coefficient differential operators, that is,
\begin{align}\label{zhuan-huan}\nonumber
  &\left(\pa_t^2-\sum_{j,k=1}^n\frac{\pa}{\pa x_j}g^{jk}(x)\frac{\pa}{\pa x_k}\right)E_0(t,d_g(0,x))\\\nonumber
  =&\left(\pa_t^2-\sum_{j,k=1}^n\frac{\pa}{\pa x_j}g^{jk}\frac{\pa}{\pa x_k}\right)E_0(t,d_g(0,x))\\
  =&|g(0)|^{-\frac12}\delta_{0,0}(t,x).
\end{align}

\begin{lemma}\label{V-C}
  Let $F\in C^\infty(\R)$. Thus for  $x$ near $0$, we have
  \begin{align*}
  \Big(\sum_{j,k}\pa_jg^{jk}(0)\pa_k\Big)F\Big(\sum_{l,m}g_{lm}(0)x_lx_m\Big)=\Big(\sum_{j,k}\pa_jg^{jk}(x)\pa_k\Big)F\Big(\sum_{l,m}g_{lm}(0)x_lx_m\Big)
\end{align*}
\end{lemma}
\begin{proof}
  The lemma will follows the strong identity, i.e.,
  \begin{align}\label{S-VC}\nonumber
(\sum_{k=1}^ng^{jk}(0)\pa_k)F(\sum_{l,m}g_{lm}(0)x_lx_m)=&(\sum_{k=1}^ng^{jk}(0)\pa_k)F(\sum_{l,m}g_{lm}(0)x_lx_m)\\
=&2x_jF'(\sum_{l,m}g_{lm}(0)x_lx_m).
\end{align}
By the chain rule, we can get
\begin{align}
  &\Big(\sum_{k=1}^ng^{jk}(0)\pa_k\Big)F\Big (\sum_{l,m}g_{lm}(0)x_lx_m\Big)\\
  =&F'\Big(\sum_{l,m}g_{lm}(0)x_lx_m\Big )\times2\sum_{k,l}g^{jk}(0)g_kl(0)x_l\\
  =&2x_jF'\Big(\sum_{l,m}g_{lm}(0)x_lx_m\Big),
\end{align}
hence the strong identity \eqref{S-VC} holds. Therefore, the lemma follows at once.
\end{proof}

Since $\square E_\nu=\nu E_{\nu-1}$ for $\nu=1,2,...$ and Lemma \ref{V-C}, we can acquire
\begin{align}\label{Identity-1}
  \Big(\pa_t^2-\sum_{j,k=1}^{n}\frac{\pa}{\pa x_j}g^{jk}(x)\frac{\pa}{\pa x_k}\Big)E_\nu(t,d_g(0,x))=\nu E_{\nu-1}(t,d_g(0,x)),\ \ \nu=1,2,...
\end{align}
Fix $t$ and let $F_\nu(|x|^2)=E_\nu(t,|x|)$, we can gain
$$\Big(\sum_{k}g^{jk}(x)\pa_k\Big)E_\nu(t,d_g(0,x))=2x_j(\pa_r F_\nu)(d_g^2(0,x)),$$
Using the fact $\pa_r E(t,r)=2rF'(r^2)$ and Proposition \ref{funcdamential}, one has
$$4\pa_rF(r^2)=-E_{\nu-1},$$
thus
\begin{align}\label{Identity-2}
-2\Big(\sum_{k}g^{jk}(x)\pa_k\Big)E_\nu(t,d_g(0,x))=x_jE_{\nu-1}(t,d_g(0,x)), \ \ \nu=1,2,...
\end{align}
and
\begin{align}\label{Identity-3}
  \Big(\sum_{k}g^{jk}(x)\pa_k\Big)E_0(t,d_g(0,x))=2x_jF_0'(d_g^2(0,x)).
\end{align}
Using the identities \eqref{Identity-2} and \eqref{Identity-3}, it is equivalent to
$$\pa_kE_v(t,d_g(0,x))=-\frac12\Big(\sum_{j}g_{jk}(0)x_j\Big)E_{\nu-1}(t,d_g(0,x)),$$
and
$$\pa_kE_0(t,d_g(0,x))=2\Big(\sum_{j}g_{jk}(0)x_j\Big)F_0'(t,d_g(0,x)).$$

Now, we construct  the Hadamard parametrix for wave kernel following the similar argument in \cite{Sogge}.
\begin{theorem}\label{W-F}
  Let $\Omega$ be a fixed open and relatively compact subset of $M$. Then there is a constant $\delta>0$ and functions $\alpha_\nu\in C^\infty(M\times \Omega)$ so that when $t<\delta$ we have
\begin{align}\nonumber
&(\pa_t^2+Q)\sum_{\nu=0}^N\alpha_\nu(x,y)E_\nu(t,d_g(x,y))\\
=&|g(y)|^{-\frac12}\delta_{0,y}(t,x)+(Q\alpha_N(x,y))E_N(t,d_g(x,y)),\quad (x,y)\in M\times \Omega
\end{align}
with
$$Q:=(i\nabla_{g}+{\bf \tilde{A}})^2+\frac{(n-2)^2}4,$$
and $\alpha_0(y,y)=1.$
\end{theorem}
\begin{proof}
 Note first that
\begin{align*}
 Q=&-\Delta_g+2i\sum_{k=1}^nA_k\pa_k+i{\rm div}{\bf \tilde{A}}+|{\bf \tilde{A}}|_g^2+\frac{(n-2)^2}4\\
  =&-\sum_{j,k=1}^n\pa_jg^{jk}(x)\pa_k+\sum_{k=1}^na_k\pa_k\\
  &+2i\sum_{k=1}^nA_k\pa_k+i{\rm div}{\bf \tilde{A}}+|{\bf \tilde{A}}|_g^2+\frac{(n-2)^2}4\\
  =&-\sum_{j,k=1}^n\pa_jg^{jk}(x)\pa_k+\sum_{k=1}^nb_k\pa_k+B,
\end{align*}
where
$$a_k=-\sum_{j=1}^ng^{jk}\pa_j(|g(x)|^{\frac12}),\ b_k=a_k+2iA_k,$$
and
$$B=i{\rm div}{\bf \tilde{A}}+|{\bf \tilde{A}}|_g^2+\frac{(n-2)^2}4.$$

Let $\alpha_0\in C^\infty$, thus we can get
\begin{align}
  (\pa_t^2+Q)\alpha_0E_0(t,d_g(0,x))=&\alpha_0(0)|g|^{-\frac12}+(Q\alpha_0)E_0(t,d_g(0,x))\\
  &+2\sum_{j,k=1}^ng_{jk}(0)b_k(x)x_j\alpha_0F_0'(d_g^2(0,x))\\
  &-4\sum_{j=1}^nx_j\pa_j\alpha_0F_0'(d_g^2(0,x))\\
  =&\alpha_0(0)|g|^{-\frac12}+(Q\alpha_0)E_0(t,d_g(0,x))\\
  &+2(\rho\alpha_0-2\langle x,\nabla_x\alpha_0\rangle)F_0'(d_g^2(0,x)),
\end{align}
if we set
$$\rho(x)=\sum_{j,k=1}^ng_{jk}(0)b_kx_j=\sum_{j,k=1}^ng_{jk}(x)b_kx_j.$$
Similarly, for $\nu=1,2,3,...$, one has
\begin{align}
(\pa_t^2+Q)\alpha_0E_\nu(t,d_g(0,x))=&\nu\alpha_\nu E_{\nu-1}(t,d_g(0,x))+(Q\alpha_0)E_\nu(t,d_g(0,x))\\
&-\frac12((\rho\alpha_\nu-2\langle x,\nabla_x\alpha_\nu\rangle))E_{\nu-1}(t,d_g(0,x)).
\end{align}

Next, we will choose $\alpha_0(x)$ so that it solves the transport equation
\begin{align}\label{Tran-1}
\rho \alpha_0=2\langle x,\nabla_x\alpha_0\rangle,\ \ \alpha_0(0)=1
\end{align}
and for $\nu\geq1$
\begin{align}\label{Tran-2}
2\nu\alpha_\nu-\rho\alpha_\nu+2\langle x,\nabla_x\alpha_\nu\rangle+2Q\alpha_{\nu-1}=0.
\end{align}
Using polar coordinations, $x=r\omega,$ thus \eqref{Tran-1} can be written
$$\pa_r\alpha_0=\frac{\rho\alpha_0}{2r}.$$
Since $\rho(0)=0$, we have
$$\alpha_0(x)=\exp\Big(\int_0^1\frac12\rho(sx)\frac{ds}s\Big).$$
Let $\Theta(x):=|g(x)|^{\frac12}$, hence we have
${\rm Re}\rho(x)=-r\Theta^{-1}\pa_r(\Theta)$, and then we can obtain
$$\alpha_0(x)=|g(x)|^{-\frac14}|g(0)|^{\frac14}\exp i\Big(\int_0^1\frac12{\rm Im}\rho(sx)\frac{ds}s\Big).$$

Let $u=\alpha_\nu/\alpha_0$ and $f(x)=-Q\alpha_{\nu-1}/\alpha_0$, then we can rewrite \eqref{Tran-2} as
$$\nu u+r\pa_r u=f,$$
and since $\pa_r(r^\nu u)=r^{\nu-1}f.$ Consequently, we have
$$r^{\nu}u(r\omega)=r^\nu\int_0^1s^{\nu-1}f(sx)ds.$$
In other words
$$\alpha_\nu =\alpha_0\int_0^1t^{\nu-1}\frac{-Q\alpha_{\nu-1}(tx)}{\alpha_0(tx)}dt,\ \nu=1,2,3,...$$
By induction, the $\alpha_\nu$ given by this formula are smooth and unique.
\end{proof}

Using the energy conservation and Sobolev embedding, we can get the follow kernel function expression for wave equation.
\begin{lemma}[Wave kernel \cite{Sogge}]\label{Wave}
Let $d_g(\theta,\omega)<\varepsilon_0$ with $\varepsilon_0$ be the injective radius of $Y$ and suppose $\varepsilon_0>\pi$. Let $t>0$, then for $N>n+3$, we have
\begin{align}
  \cos(t\nu)(\theta,\omega)=K_N(t,\theta,\omega)+R_N(t,\theta,\omega),
\end{align}
where $R_N(t,\theta,\omega)\in C^{N-n-3}([0,\pi]\times Y\times Y)$ and
\begin{align}
K_N(t,\theta,\omega)&=\sum_{k=0}^N\alpha_k(\theta,\omega)\pa_tE_k(t,d_h(\theta,\omega))
=\sum_{k=0}^N\alpha_k(\theta,\omega)t
\frac{(t^2-d_h^2(\theta,\omega))_+^{k-\frac{n}2}}{\Gamma(k-n/2+1)}
\end{align}
with ${\rm Re}\alpha_0(\theta,\omega)=|h(\theta)|^{-\frac14}|h(\omega)|^{\frac14}$.
\end{lemma}

\subsection*{Acknowledgements}
X. F. Gao was supported by National Natural Science Foundation of China (No.12301124), China Postdoctoral Science Foundation (No.2024M762972) and Postgraduate Education Reform and Quality Improvement Project of Henan Province (No.YJS2024JC07). C. B. Xu was partially supported by National Natural Science Foundation of China (No.12401296) and  Qinghai Natural Science Foundation (No.2024-ZJ-976).

 \begin{center}

\end{center}

\end{document}